\documentclass[11pt]{article}

\usepackage{appendix}
\usepackage{mathtools}
\usepackage[T1]{fontenc}
\usepackage{amsfonts}
\usepackage{amsmath}
\usepackage{amssymb}
\usepackage{amsthm}
\usepackage{bbm}
\usepackage{bm}
\usepackage{mathrsfs}
\usepackage{xcolor} 
\usepackage{pdfsync}
\usepackage{enumitem}

\usepackage{tikz}

\newcommand{\RR}{\mathbb{R}}
\newcommand{\R}{\RR}
\newcommand{\Z}{\mathbb{Z}}

\newcommand{\NN}{\mathbb{N}}

\newcommand{\eps}{\varepsilon}

\usepackage[margin=31mm]{geometry}
\newcommand{\mykill}[1]{}

\newcommand{\e}{\varepsilon}
\newcommand{\dd}{\mathrm d}

\usepackage[colorlinks=true,
            linkcolor=blue,
            citecolor=blue,
            urlcolor=blue]{hyperref}

\usepackage[nameinlink,capitalize,noabbrev]{cleveref}

\theoremstyle{plain}
\newtheorem{theorem}{Theorem}[section]
\newtheorem{proposition}[theorem]{Proposition}
\newtheorem{lemma}[theorem]{Lemma}
\newtheorem{corollary}[theorem]{Corollary}

\theoremstyle{definition}

\newtheorem{remark}[theorem]{Remark}

\crefname{theorem}{theorem}{Theorems}
\Crefname{theorem}{Theorem}{Theorems}

\crefname{lemma}{lemma}{lemmas}
\Crefname{lemma}{Lemma}{Lemmas}

\crefname{proposition}{proposition}{propositions}
\Crefname{proposition}{Proposition}{Propositions}

\crefname{corollary}{Corollary}{corollaries}
\Crefname{corollary}{Corollary}{Corollaries}

\crefname{definition}{definition}{definitions}
\Crefname{definition}{Definition}{Definitions}

\crefname{remark}{remark}{remarks}
\Crefname{remark}{Remark}{Remarks}

\crefname{example}{example}{examples}
\Crefname{example}{Example}{Examples}

\crefname{assumption}{assumption}{assumptions}
\Crefname{assumption}{Assumption}{Assumptions}

\crefname{section}{section}{sections}
\Crefname{section}{Section}{Sections}

\crefname{equation}{equation}{equations}
\Crefname{equation}{Equation}{Equations}

\newlist{myenum}{enumerate}{3}
\setlist[myenum,1]{label={\rm (H\arabic*)},
                   ref  ={\rm (H\arabic*)}}

\crefname{myenumi}{property}{properties}
\Crefname{myenumi}{Property}{Properties}

\renewenvironment{thebibliography}[1]{%
\begin{oldthebibliography}{#1}%
\setlength{\baselineskip}{.9em}
\linespread{1}
\small
\setlength{\parskip}{.40ex}%
\setlength{\itemsep}{.30em}%
}%
{%
\end{oldthebibliography}%
}

\newcommand{\supp}{\mathrm{supp}}
\NewDocumentCommand{\esssup}{o}{%
  \operatorname*{ess\,sup}_{\IfValueT{#1}{#1}}%
}
\begin{document}

\title{\Huge Sharp local sparsity of regularized optimal transport}
\date{\today}
\author{  
  Alberto Gonz{\'a}lez-Sanz%
  \thanks{Department of Statistics, Columbia University, \textcolor{blue}{ag4855@columbia.edu}} \and  Rishabh S.  Gvalani \thanks{School of Mathematics, University of Edinburgh, \textcolor{blue}{rgvalani@ed.ac.uk}} \and Lukas Koch \thanks{Department of Mathematics, University of Sussex,  \textcolor{blue}{lukas.koch@sussex.ac.uk}}
  }
  
\maketitle \vspace{-1.5em}
\begin{abstract}

In recent years, the use of entropy-regularized optimal transport with $L^p$-type entropies has become increasingly popular. In this setting, the solutions are sparse, in the sense that the support of the regularized optimal coupling, $\mathrm{supp}(\pi_\varepsilon)$, shrinks to the support of the original optimal transport problem as $\varepsilon \to 0$.

The main open question concerns the rate of this convergence. In this paper, we obtain sharp local results away from the boundary. We prove that the supports $\mathrm{supp}(\pi_\varepsilon(\cdot \mid x))$ of the conditional measures, $\pi_\varepsilon(\cdot \mid x)$, behave like balls of radius $\e^\frac 1 {d(p-1)+2}$. This allows us to show that the regularized potentials are uniformly strongly convex and to derive the rate of convergence of these potentials toward their unregularized limit. Our results generalize the results of (Gonz\'alez-Sanz and Nutz, SIAM J.~Math.~Anal.) and (Wiesel and Xu, Ibid) to the multivariate case and beyond the case of self-transport.

\end{abstract}

 \vspace{1em}

{\small

}
 \vspace{0em} 
 \section{Introduction}
The optimal transport (OT) problem between two probability measures $\lambda$ and $\mu$ is defined as (cf.~\cite{villani2008optimal}) 
\begin{align}\label{OT}\tag{OT}
{\rm OT}:= \inf_{\pi\in \Pi(\lambda,\mu)} \int \frac{1}{2}\|x-y\|^2 \;\dd\pi, 
 \end{align}
 where $\pi\in \Pi(\lambda,\mu)$ if $\pi(\cdot \times \R^d)=\lambda$ and $\pi(\R^d\times \cdot )=\mu$.  
In this paper  we consider the problem
\begin{align}\label{ROT}\tag{ROT}
{\rm ROT}_{\eps,p}:= \inf_{\pi\in \Pi(\lambda,\mu),  \pi\ll \lambda\otimes \mu} \int \frac{1}{2} \|x-y\|^2 \;\dd\pi + \e \int h_p\left(\frac{\dd\pi}{\dd(\lambda\otimes \mu)}\right) \dd(\lambda\otimes \mu), 
 \end{align}
 where $h_p(z) = \frac{|z|^p-1}{p-1}$ for $p\in (1,2]$, $\lambda\otimes \mu$ denotes the product measure and $\ll$ stands for absolute continuity. A minimizer of \eqref{ROT} (resp.~\eqref{OT})  is called ROT  plan (resp.~OT plan) and denoted by $\pi_\eps$ (resp.~$\pi_0$).  
Recently, ROT has gained popularity as an alternative to entropic optimal transport (EOT). 
Recent results show that ROT can be computed in linear time \cite{GonzalezSanzNutzRiveros.25} and avoids the curse of 
dimensionality \cite{GonzalezSanzEcksteinNutz.25,GonzalezSanzDelBarrioNutz.25}---exactly the same as EOT \cite{MenaWeed.2019.Nips}. However, the transport plans of EOT and ROT 
are quite different. While the EOT plan has maximal support, ROT plans have sparse support \cite{zhang.2023.manifoldlearningsparseregularised}. 
This is fundamentally because ROT plans $\pi_\eps$ have density with respect to the product measure 
given by (see \cite{Nutz.QOT.2024,BayraktarEckstein.2025.BJ,GonzalezSanzEcksteinNutz.25})
\begin{equation} \label{eq:density-ROT}
\rho_{\varepsilon}(x,y)=\frac{d \pi_\eps}{ d(\lambda \otimes \mu)}(x,y) = \frac{1}{\eps^{q-1}q^{q-1}}\left( f_{\varepsilon}(x) + g_{\varepsilon}(y) - \frac{1}{2}\|x-y\|^2\right)_{+}^{q-1},
\end{equation}
where \((f_{\varepsilon}, g_{\varepsilon})\) are the solutions of the dual problem 
\begin{equation}\label{Dual} \tag{D-ROT}
{\rm ROT}^*_{\eps,p}=\sup_{a,b}\int a(x)+b(y)-\frac{1}{\eps^{q-1} q^q}
\left({a(x)+b(y)-\frac12\|x-y\|^2}\right)_+^{q}\dd\lambda(x)\dd\mu(y).\\
\end{equation}
Here $q$ denotes the Young conjugate of $p$, i.e., $q=\frac{p}{p-1}$. 
In EOT, the solutions are of the form 
\begin{equation} \label{eq:density-EOT}
\pi_{\varepsilon}^{EOT} (x,y) \dd x \dd y = e^{\frac{f_{\varepsilon}^{EOT}(x) + g_{\varepsilon}^{EOT}(y) - \frac{1}{2}\|x-y\|^2}{\eps} }\, \dd(\lambda \otimes \mu)(x,y) 
\end{equation}
with \((f_{\varepsilon}^{EOT}, g_{\varepsilon}^{EOT})\) being the dual solutions of EOT. The positive part in 
\eqref{eq:density-ROT} allows the density to vanish and the support to narrow. In view of \eqref{eq:density-EOT}, the density of the EOT plan cannot vanish. Indeed, under some conditions (see \cite{Nutz.QOT.2024}) the support of $\pi_\eps$ converges to the support of the OT plan. In this paper we  characterize the rate of convergence at which the diameters of the sections 
$$\mathcal{S}_x=  \{y: \rho_\eps(x,y)>0\} \quad \text{and} \quad  \mathcal{T}_y=  \{x: \rho_\eps(x,y)>0\}  $$  decrease. This problem has been studied by \cite{GonzalezSanzNutz2024.Scalar} and \cite{WieselXu.24},  where the sharp rates have been obtained for univariate data and general marginals as well as multivariate data but  $\mu=\lambda$. In this paper we obtain, under mild conditions, that, for $x$ in the interior ${\rm int}(\Omega_0)$ of the support $\Omega_0$ of $\lambda$,
$$ \mathbb{B}\left(\nabla \varphi_\e(x),\frac 1 {R_0} \e^\frac 1 {d(p-1)+2}\right)\subset \mathcal S_x\subset \mathbb{B}\left(\nabla \varphi_\e(x), R_0 \e^\frac 1 {d(p-1)+2}\right) , $$
where $R_0$ depends on the distance from $x$ to the boundary of $\Omega_0$. Here
\begin{align}\label{eq:derivative}
\begin{cases}
    \nabla \varphi_\e(x) = \frac{\int_{\mathcal S_x} y\xi(x,y)^{q-2}\;\dd\mu(y)}{\int_{\mathcal S_x} \xi(x,y)^{q-2}\;\dd\mu(y)},\\
     \nabla \psi_\e(y) = \frac{\int_{\mathcal T_y} x\xi(x,y)^{q-2}\;\dd\lambda(x)}{\int_{\mathcal T_y} \xi(x,y)^{q-2}\;\dd\lambda(x)},
\end{cases} \quad {\rm for}\  \xi(x,y)=\langle x,y \rangle-\varphi_\e(x) -\psi_\e (y),  
\end{align}
are the gradients of the convex functions $(\varphi_\e,\psi_\e) =(\|\cdot\|^2/2-f_\eps,\|\cdot\|^2/2- g_\eps)$ (cf.~\cite{WieselXu.24} or \cite{GvalaniKoch2026} for the convexity of $\varphi_\e$ and $\psi_\e$).  This result is  \Cref{thm:sparsity}, which is proved in \Cref{Section:proof-main}. In light of the examples in \Cref{sec:Explicit}, the rates we obtain are sharp. Our proof relies on the uniform bound on 
$ \|\nabla^2 \varphi_\e \|_{L^\infty_{\rm loc}}\leq C$ derived in \cite{GvalaniKoch2026}. We underline that, with our proof techniques, a uniform bound  $ \|\nabla^2 \varphi_\e \|_{L^\infty}\leq C$ would generalize most of our local-type results  up to the boundary.

Our second contribution is to show that the ROT convex potentials  $ \varphi_\e $ are uniformly strongly convex in ${\rm int}(\Omega_0)$. In particular, \Cref{Corollary:InteriorStrongConvexity} states that, for all $x\in {\rm int}(\Omega_0)$,
$$ \inf_{\|h\|=1} \langle \nabla^2 \varphi_\e(x) h ,h \rangle \geq \frac{1}{C}, $$
where $C$ depends on the distance from $x$ to the boundary of $\Omega_0$. Our last contribution (cf.~\Cref{Coro:Rates-map}) shows that 
for every $K_0\Subset \Omega_0$, there exists a constant $C>0$ such that, for every $\eps\in (0,1]$,  
$$ \|\nabla \varphi_\eps- \nabla \varphi\|_{L^2(K_0)}   \leq \|  {\rm d}(\mathcal{S}_{(\cdot)}, \nabla \varphi) \|_{L^2(K_0)} \leq C \e^\frac 1 {d(p-1)+2},   $$ 
 where $\nabla \varphi$ denotes the OT map from $\lambda$ to $\mu$ and ${\rm d}(\mathcal{S}_{(\cdot)}, \nabla \varphi)$ denotes the distance between $\mathcal{S}_{(\cdot)}$ and $\nabla \varphi$. 
\paragraph{Organization} The remainder of the paper is organized as follows. The notation section is \Cref{Section:notation}. The statements of our main result and its corollaries are presented in  \Cref{Section:main-results}.  \Cref{Sect:preliminar}  states some known preliminary results.  In  \Cref{Section:proof-main} we prove \Cref{thm:sparsity} and in \Cref{Section:proof-of-corollary} we prove \Cref{Corollary:InteriorStrongConvexity}. Finally, we provide explicit solutions for \eqref{ROT} in the case of self-transport on the torus in \Cref{sec:Explicit}. Omitted proofs are contained in \Cref{appendix}.
\section{Notation}\label{Section:notation} We write $A\Subset B $ if there exists a compact set $K$ such that $A\subset K \subset {\rm int} \, B$, where ${\rm int}\, B$ stands for the Euclidean interior. The Euclidean closure of $A$ is denoted as $\overline{A}$.  A smooth domain is a set $A$ with $\mathcal{C}^\infty$ boundary $\partial A$. The open ball of center $x$ and radius $\delta>0$ is denoted by $\mathbb{B}(x,\delta)$. The support of a probability measure $\mu$ is denoted by ${\rm supp}\, \mu$.   We recall that the \emph{essential supremum of $f$ on $K_0$} is defined by
\[
\operatorname*{ess\,sup}_{x\in K_0} f(x)
:= \inf \Big\{ M \in \mathbb{R} \cup \{+\infty\} \;:\;
f(x) \le M \ \text{for Lebesgue-a.e. } x \in K_0 \Big\}.
\]
We use the convention $(t)_+^0={\bf 1}_{t\geq 0}$. For $p\in (1,\infty)$, we set $q=p^\prime$, where $\frac 1 p+\frac 1 {p^\prime}=1$. 

In this work, we will use the notation $A_\eps \lesssim B_\eps$ to mean that there exists a constant $C$, independent of $\eps$, such that $A_\eps \leq C B_\eps$ for $\eps$ small enough. We will use $A_\eps \approx B_\eps$ when $A_\eps \lesssim B_\eps$ and $B_\eps \lesssim A_\eps$. For a set $E \subset \mathbb{R}^d$, we define its diameter by $\operatorname{diam}(E) := \sup\{ |x-y| : x,y \in E\}.$

Given a probability measure $\mu$ on $\mathbb{R}^d$, we define its support, denoted by $\operatorname{supp}(\mu)$, as the smallest closed set $F \subset \mathbb{R}^d$ such that $\mu(\mathbb{R}^d \setminus F)=0.$

\section{Main results}\label{Section:main-results}
In this paper, we assume that $\lambda$ and $\mu$ are probability measures with compact supports $\Omega_0$ and $\Omega_1$, respectively. We denote
\[
\Omega_0 := \supp\, \lambda, 
\qquad 
\Omega_1 := \supp\, \mu.
\]
We further assume that both $\lambda$ and $\mu$ admit $C^{0,\alpha}$ densities, for some $\alpha>0$, which are bounded above and bounded away from zero on their supports. Finally, we assume that the optimal transport map $T = \nabla \varphi : \Omega_0 \to \Omega_1$ between $\lambda$ and $\mu$ is bi-$C^{1,\alpha}$, that is, $T$ is a $C^{1,\alpha}$ diffeomorphism with $C^{1,\alpha}$ inverse.  In particular, due to \cite{Caffarelli,Chen.et.al.2021.AoM}, all our results hold when $\Omega_0,\Omega_1$ are $\mathcal{C}^2$ convex domains.

Our main result, showing the precise sparsity for the section $\mathcal{S}_x=\{y: \rho_\eps(x,y)>0\}$   is the following. The same result holds for $  \mathcal{T}_y=  \{x: \rho_\eps(x,y)>0\} $.
\begin{theorem}[Interior sharp sparsity]\label{thm:sparsity}
For any smooth domain $K_0\Subset \Omega_0$, there exists $R_0=R_0(K_0)$ and $\e_0=\e_0(K_0)$ such that for every $x\in K_0$ and $\e\in (0,\e_0]$,
\begin{align}\label{eq:mainInequality}
\mathbb{B}\left(\nabla \varphi_\e(x),\frac 1 {R_0} \e^\frac 1 {d(p-1)+2}\right)\subset \mathcal S_x\subset \mathbb{B}\left(\nabla \varphi_\e(x), R_0 \e^\frac 1 {d(p-1)+2}\right). 
\end{align}
If $\Omega_1$ is convex, the result holds for any $\e\in (0,1]$.
\end{theorem} 

\Cref{thm:sparsity} states that the sections $\mathcal S_x$ behave like balls of radii approximately $\e^\frac 1 {d(p-1)+2}$. In light of the examples in \Cref{sec:Explicit} this rate is sharp. The proof uses the estimate
$ \esssup[K_0] |\nabla^2 \varphi_\e | \lesssim 1$ shown in  \cite{GvalaniKoch2026}. 
In the following result we show that the restriction $\varphi_\eps\vert_{K_0}$ of  $\varphi_\eps$ to $K_0\Subset \Omega_0$ is uniformly strongly convex as $\eps\to 0$.  
\begin{corollary}[Interior strong convexity]\label[corollary]{Corollary:InteriorStrongConvexity}
    Let the assumptions of  \Cref{thm:sparsity} hold. Then for every $K_0\Subset \Omega_0$ there exists  $C,\eps_0>0$, such that $\varphi_\eps\vert_{K_0} $ is $C$-strongly convex on any convex subset of $K_0$  for all $\eps\in (0,\eps_0]$. 
\end{corollary}

\Cref{thm:sparsity} describes completely how the support contracts as $\varepsilon$ decreases. From here it is easy to derive that the Lebesgue measure  $|\mathcal{S}_x|$ of $\mathcal{S}_x$ is of order $\eps^{\frac{d}{d+2}}$. 
We now focus on analyzing how it shrinks toward the support of the transport plan $\pi_0$. 
Recall that the support of $\pi_0$ is contained in the graph of the gradient map $\nabla \varphi$.  We follow  a strategy very similar to that of~\cite[Chapter~6]{GonzalezSanzNutz2024.Scalar}. 
More precisely, we first control the projection onto the barycenter using the arguments of~\cite[Section~4.2]{CarlierPegonTamanini.22}. 
We then exploit the sparsity of the support.
Define 
$$ T_{\e}(x) =\frac{\int y (\xi(x,y))_+^{q-1} \dd\mu(y)}{ \int (\xi(x,y))_+^{q-1} \dd\mu(y)} = \int y \pi_{\eps}(y|x),$$
where $\pi_\eps(y|x)$ denotes the conditional measure of the ROT plan of $y$ given $x$. The same argument employed by \cite{CarlierPegonTamanini.22} or \cite{GonzalezSanzNutz2024.Scalar} yields 
\begin{equation}
    \label{eq:a-la-Carlier}
     \|T_{\e}- \nabla \varphi\|_{L^2(\Omega_0)}^2 \lesssim   \int \|x-y\|^2 d(\pi_\eps-\pi_0)(x,y) \leq {\rm ROT}_\eps-{\rm OT} \lesssim \e^\frac 2 {d(p-1)+2},
\end{equation}
where the last inequality follows from \cite{Eckstein.Nutz.2023}. 
Let $ {\rm d}(\mathcal{S}_{(\cdot)}, \nabla \varphi)  $ denote the mapping $x\mapsto \sup_{y\in \mathcal{S}_{x}} \| y- \nabla \varphi(x)\|  $. 
Since $\mathcal{S}_x$ is convex and  $T_{\e}(x)$ is an average on $\mathcal{S}_x$,  $T_{\e}(x)\in \mathcal{S}_x$. 
Hence, for $K_0\Subset \Omega_0$,  we get  
$$  \| {\rm d}(\mathcal{S}_{(\cdot)}, \nabla \varphi) \|_{L^2(K_0)} \leq \| T_\eps- \nabla \varphi \|_{L^2(K_0)}+\sup_{x\in K_0}{\rm diam}(\mathcal{S}_x) \lesssim  \e^\frac 1 {d(p-1)+2} ,  $$
where we used  \Cref{thm:sparsity} and \eqref{eq:a-la-Carlier}. 
Since  
$  \nabla \varphi_\eps(x)= \frac{\int_{\mathcal S_x} y (\xi(x,y))_+^{q-2} \dd \mu(y)}{ \int_{\mathcal S_x} (\xi(x,y))_+^{q-2} \dd \mu(y)} \in \mathcal{S}_x, $ we have shown the following result.  
\begin{corollary}[Rates of ROT map]\label[corollary]{Coro:Rates-map}
  Fix $K_0\Subset \Omega_0$. Then there exists a constant $C>0$ and $\e_0>0$ such that, for every $\eps\in (0,\e_0]$,  
$$ \|\nabla \varphi_\eps- \nabla \varphi\|_{L^2(K_0)}   \leq \|  {\rm d}(\mathcal{S}_{(\cdot)}, \nabla \varphi) \|_{L^2(K_0)} \leq C \e^\frac 1 {d(p-1)+2}.   $$ 
If $\Omega_1$ is convex, the result holds for all $\e\in (0,1]$.
\end{corollary}
\section{Preliminaries}\label{Sect:preliminar}
\subsection{Regularized optimal transport}
Recall that $\lambda$ and $\mu$ are probability measures with compact supports $\Omega_0$ and $\Omega_1$, respectively and $\mathcal{C}^{0,\alpha}$ densities bounded away from zero and infinity. We recall our convention $(t)_+^0={\bf 1}_{t\geq 0}$ and the following result proved in \cite[Proposition~2.3]{GonzalezSanzEcksteinNutz.25}. 
\begin{proposition}\label[proposition]{propo:ROTprelims}
    Let $\lambda,\mu$ satisfy our assumptions.
    \begin{enumerate}
    \item
    The strong duality 
      ${\rm ROT}_{\eps,p}={\rm ROT}_{\eps,p}^*$
    holds. 
    \item
    The primal problem~\eqref{ROT} admits a unique optimizer $\pi_\eps\in \Pi( \lambda,\mu)$.
    \item
    The dual problem~\eqref{Dual} admits a (non-unique) optimizer $(f_\eps,g_\eps)\in L^\infty(\lambda)\times L^\infty(\mu)$.
    \item
    A pair $(f_\eps,g_\eps)\in L^\infty(\lambda)\times L^\infty(\mu)$ is an optimizer of the dual problem~\eqref{Dual} if and only if there exists  a version of $(f_\eps,g_\eps)$ such that
    \begin{equation}
        \label{eq:Srodinger-system}
        \begin{cases}
             \int \left( f_{\varepsilon}(x) + g_{\varepsilon}(y) - \frac{1}{2}\|x-y\|^2\right)_{+}^{q-1} \dd\mu(y)=\eps^{q-1}q^{q-1}\quad \text{for all }x\in\Omega_0,\\
             \int \left( f_{\varepsilon}(x) + g_{\varepsilon}(y) - \frac{1}{2}\|x-y\|^2\right)_{+}^{q-1} \dd\lambda(x)=\eps^{q-1}q^{q-1}\quad \text{for all }y\in\Omega_1.
        \end{cases}
    \end{equation}
    For such a version $(f_\eps,g_\eps)$, there exists a constant $C$ depending on $
    \Omega_i$ and $p$ such that  $\|f_\eps \oplus g_\eps \|_\infty \leq C$ holds.
    Moreover, $f_\eps$ and $g_\eps$ are uniformly Lipschitz in $\Omega_0$ and $\Omega_1$. 
    \item Let $(f_\eps,g_\eps)$ solve \eqref{eq:Srodinger-system} and abbreviate $\xi(x,y):=f_{\varepsilon}(x) + g_{\varepsilon}(y) - \frac{1}{2}\|x-y\|^2$. 
    Then there exists $\delta=\delta(\eps) > 0$ such that
     \begin{align*}
        \int \left(\xi(\cdot,y)\right)_{+}^{q-2}  \dd \mu(y)\geq \delta \quad\mbox{on }\Omega_0, \qquad
        \int \left(\xi(x,\cdot)\right)_{+}^{q-2}  \dd \lambda(x)\geq \delta\quad\mbox{on }\Omega_1.
     \end{align*}
\end{enumerate}
\end{proposition}
Now we show higher order of differentiability. It is convenient to work instead with 
$$ \varphi_\e=\frac{1}{2}\|\cdot\|^2-f_\eps \quad {\rm and}\quad \psi_\e = \frac{1}{2}\|\cdot\|^2-g_\eps, $$
which are convex functions on convex subsets of $\Omega_0$ and $\Omega_1$, respectively, as the following proposition shows. 
\begin{proposition}\label{Proposition-convex-and-first-derivatives}
It follows that $\varphi_\e$ is convex and $\mathcal{C}^{1}$ with gradient
\begin{equation}
    \label{eq:gradient-propo}
     \nabla \varphi_\e(x) = \frac{\int_{\mathcal{S}_x} y (\xi(x,y))^{q-2}\;\dd\mu(y)}{\int_{\mathcal{S}_x}  (\xi(x,y))^{q-2}\;\dd\mu(y)}
\end{equation}
on any convex subset of $\Omega_0$. 
As a consequence, for all $x\in \Omega_0$, if $\mathcal S_x\cap\partial \Omega_1=\emptyset$ or $\Omega_1$ is convex, $\mathcal{S}_x$ is a convex set.
\end{proposition}
\begin{proof}
    To show that $\varphi_\e$ is convex on convex subsets of $\Omega_0$, we follow the same strategy as in the proof of  \cite[Lemma~2.4]{WieselXu.24} for $p=2$. Fix $t\in (0,1)$ and $x\neq z$. Then the convexity of $(\cdot)_+^{q-1}$ gives
    \begin{align*}
       &\int  \left( \langle (1-t) x+t z,y \rangle -(1-t)\varphi_{\varepsilon}(x) -t\varphi_{\varepsilon}(z)- \psi_{\varepsilon}(y) \right)_{+}^{q-1} \dd\mu(y)\\
       &= \int  \left((1-t) ( \langle x,y \rangle -\varphi_{\varepsilon}(x) -\psi_{\varepsilon}(y))+t( \langle z,y \rangle -\varphi_{\varepsilon}(z) -\psi_{\varepsilon}(y)) \right)_{+}^{q-1} \dd\mu(y)\\
       &\leq (1-t) \int  \left( \langle x,y \rangle -\varphi_{\varepsilon}(x) -\psi_{\varepsilon}(y)\right)_{+}^{q-1} \dd\mu(y)+ t\int  \left( \langle z,y \rangle -\varphi_{\varepsilon}(z) -\psi_{\varepsilon}(y)\right)_{+}^{q-1} \dd\mu(y)\\
       &=  \eps^{q-1}q^{q-1}=  \int  \left( \langle (1-t) x+t z,y \rangle -\varphi_{\varepsilon}((1-t) x+t z) -\psi_{\varepsilon}(y) \right)_{+}^{q-1} \dd\mu(y),
    \end{align*}
    where we used also \eqref{eq:Srodinger-system}. 
    The  monotonicity of 
    $$ s\mapsto \int  \left( \langle (1-t) x+t z,y \rangle +s -\psi_{\varepsilon}(y) \right)_{+}^{q-1} \dd\mu(y),  $$
    which is strict for $s$ in a neighborhood of $\varphi_{\varepsilon}((1-t) x+t z)$,  implies that 
    $$ \varphi_{\varepsilon}((1-t) x+t z)\leq (1-t)\varphi_{\varepsilon}(x) +t\varphi_{\varepsilon}(z)   .$$
    This shows that $\varphi_{\varepsilon}$ is convex on convex subsets of $\Omega_0$. 

    Since $\varphi_{\varepsilon}$ is Lipschitz (by \Cref{propo:ROTprelims}), it is differentiable a.e.~with bounded derivative by Rademacher's theorem. Hence, \eqref{eq:gradient-propo} follows by  differentiating  both sides of \eqref{eq:Srodinger-system}.  It only remains to show that $    \nabla \varphi_\e$ is continuous. For $q=2$ this is shown in \cite{GonzalezSanzDelBarrioNutz.25} and for $q>2$ the result is straightforward as $(\xi)^{q-2}$ is $\min((q-2),1) $-H\"older continuous. 
\end{proof}
As a consequence of Reynolds' transport theorem, see \Cref{appendix}, we record the following information on second derivatives of $\varphi_\e$.
\begin{proposition} Assume $\Omega_1$ is a Lipschitz domain. Then $\varphi_\e\in W^{2,\infty}_{{\rm loc}}(\Omega_0)$ with a.e.~defined second derivative
   \begin{align}\label{eq:Form2ndDerivative}
      \langle \nabla^2 \varphi_\e(x) u,v\rangle = \begin{cases}(q-2)\frac{\int_{\mathcal S_x} \langle  \nabla \varphi_\e(x) -y, u \rangle \langle \nabla \varphi_\e(x) -y, v \rangle \xi(x,y)^{q-3}\;\mu\, \dd\mu(y)}{\int_{\mathcal S_x} \xi(x,y)^{q-2}\;\dd\mu(y)} & \text{ if } p\neq 2,\vspace{2mm}\\
\frac{\int_{\partial \mathcal S_x\setminus \partial\Omega_1} \frac{\langle  \nabla \varphi_\e(x) -y, u \rangle \langle\nabla \varphi_\e(x) -y, v \rangle }{\|\nabla \psi_\e(y)-x\|} \mu\, \dd\mathcal H^{d-1}(y)}{\mu(\mathcal S_x)}, & \text{ if } p=2.
\end{cases}
\end{align}
For any $K\Subset\Omega_1$, on $\{x\in\Omega_0 \colon \mathcal S_x\subset K\}$, the statement holds without any regularity assumption on $\Omega_1$.
\end{proposition}
\begin{proof} For $p=2$,  Reynolds transport theorem (cf.~\Cref{appendix}) allows us to compute for almost every $x\in \Omega_0$,
\begin{align*}
\langle \nabla^2 \varphi_\e(x)u,v\rangle
    &= \frac{ \int_{\partial \mathcal S_x\setminus \partial \Omega_1}  \frac{  \langle y-\nabla \varphi_\e (x), v\rangle \langle y, u\rangle   \mu(y)}{\|x-\nabla \psi_\eps(y)\|}  \dd\mathcal{H}^{d-1}(y) }{\mu(\mathcal{S}_x)}\\
    &\quad - \frac{\langle  \nabla \varphi_\e(x), u\rangle \int_{\partial \mathcal S_x}  \frac{  \langle y-\nabla \varphi_\e (x), v\rangle \mu(y)}{\|x-\nabla \psi_\eps(y)\|}  \dd\mathcal{H}^{d-1}(y) }{\mu(\mathcal{S}_x)}\\
    &= \frac{ \int_{\partial \mathcal S_x\setminus \partial \Omega_1}  \frac{  \langle y-\nabla \varphi_\e (x), v\rangle \langle y-\nabla \varphi_\e (x), u\rangle   \mu(y)}{\|x-\nabla \psi_\eps(y)\|}  \dd\mathcal{H}^{d-1}(y) }{\mu(\mathcal{S}_x)},
\end{align*}
concluding the case $p=2$. Using \Cref{appendix}, the case $p\neq 2$ follows by an analogous calculation. The interior statement holds by the same argument.
\end{proof}
We now recall the following interior regularity result for the ROT potentials.
\begin{theorem}[Interior regularity, \cite{GvalaniKoch2026}]\label{thm:interiorC2}
For any $K_0\times K_1\Subset \Omega_0\times \Omega_1$, there is $L>0$ depending only on $\Omega_0$, $\Omega_1$ such that $(\varphi_\e,\psi_\e)$ are $L$-Lipschitz. Moreover, there exist $C>0$ and $\eps_0>0$ depending only on $d,p,\Omega_0,\Omega_1$, $\|\lambda\|_{C^{0,\alpha}}$, $\|\mu\|_{C^{0,\alpha}}$, $\|T\|_{C^{1,\alpha}}$ and $\|T^{-1}\|_{C^{1,\alpha}}$ such that, for every $\eps\in (0,\eps_0)$,
\begin{align*}
\esssup[K_0] |\nabla^2 \varphi_\e |+\esssup[K_1] |\nabla^2 \psi_\e| \leq C.
\end{align*}
Finally, there exist $C>0$ and $\eps_0$ with the same dependencies such that, for every $\eps\in (0,\eps_0)$,
\begin{align*}
\esssup[x\in K_0,y\in \mathcal S_x] |\nabla^2 \psi_\e(y)|+\esssup[y\in K_1,x\in \mathcal T_y] |\nabla^2 \varphi_\e(x)|\leq C.
\end{align*}
If $\e\geq \e_0$ and $\Omega_1$ is Lipschitz, then there is $C>0$, with the same dependencies, such that $$\esssup[y\in \Omega_1] |\nabla^2 \psi_\e|\leq C.$$
\end{theorem}
\begin{proof}
With the exception of the final claim, the result is a combination of \cite[Corollary 4 and Lemma 10]{GvalaniKoch2026}.  Using \cite[Corollary~8]{GvalaniKoch2026},  we find $\e_0>0$ such that if $\e\leq \e_0$, for $(x,y)\in K_0\times \R^d \cap \supp\, \pi_\e$, $$|y-T(x)|\lesssim \e^\frac 2 {(d(p-1)+2)(d+2)}.$$ Note that $T(\partial \Omega_0)=\partial \Omega_1$ and so as $T$ is continuous and $\Omega_0$ is bounded, ${\rm d}(T(K_0),\partial \Omega_1)>0$. Reducing $\e_0$ further if necessary, it follows that $\cup_{x\in K_0} \mathcal S_x \Subset \Omega_1.$
An analogous argument for $y\in K_1$ establishes the final claim for $\e\leq \e_0$. For $\e\geq \e_0$ the claim follows from \cite[Lemma 17]{GvalaniKoch2026}.
\end{proof}

\subsection{Convex analysis}
The \emph{convex conjugate} (or \emph{Legendre--Fenchel transform}) of $f:\mathbb{R}^d \to (-\infty,+\infty]$ is the function 
$f^* : \mathbb{R}^d \to (-\infty,+\infty]$ defined by
\[
f^*(y) 
:= \sup_{x \in \mathbb{R}^d} 
\big\{ \langle x, y \rangle - f(x) \big\}.
\]
The following result is well known. It relates the boundedness of the second derivative of $f$ with the strong convexity of $f^*$. 
\begin{lemma}\label{lem:convex}
Suppose $f\colon \Omega \subset\R^d\to \R$ is convex and has second derivatives that are uniformly bounded in $\Omega$ by $L>0$ almost everywhere. Then for $x\in \Omega, y\in \R^d$ the following are equivalent:
\begin{enumerate}
\item $x\in \partial f^\ast(y)$, 
\item $\nabla f(x)=y$,
\item $f^\ast(y)+f(x)=\langle x,y\rangle$.
\end{enumerate}
Further, for any $y,y_0\in \textup{dom}\; f^\ast$ and $z\in \partial f^\ast(y_0)$, we have
\begin{align*}
f^\ast(y)\geq f^\ast(y_0)+\langle z,y-y_0\rangle +\frac 1 {2L} \|y-y_0\|^2.
\end{align*}
\end{lemma}
\begin{proof}
Setting $f=+\infty$ outside of $\Omega$,  we obtain a convex, proper and lower-semicontinuous function on $\R^d$. The first claim is \cite[Theorem 23.5]{Rockafellar.Convex}. While the second statement is certainly not new, we provide a proof for the convenience of the reader. Due to the gradient Lipschitzness of $f$, there is $L>0$ such that for any $x,z\in\Omega$ and $y\in \R^d$,
\begin{align*}
\langle y,x\rangle -f(x)\geq \langle y,x\rangle - f(z)-\langle \nabla f(z),x-z\rangle - \frac{L} 2 \|z-x\|^2.
\end{align*}
Taking supremums over $x$, and choosing $x=z+\frac{y-\nabla f(z)} L$ on the right-hand side, we find
\begin{align}
f^\ast(y)\geq& \left\langle y,z+\frac{y-\nabla f(z)} L\right\rangle-f(z)-\left\langle \nabla f(z),\frac{y-\nabla f(z)} L\right\rangle - \frac 1 {2L} \|y-\nabla f(z)\|^2\notag \\
=& \langle y_0,z\rangle - f(z)+\langle y-y_0,z\rangle+\frac 1 L \|y-\nabla f(z)\|^2-\frac 1 {2L} |y-\nabla f(z)|^2,\label{eq:bound-f-star}
\end{align}
for every $y_0\in \R^d$. In particular, taking
 $z\in \partial f^\ast(y_0)$ in   \eqref{eq:bound-f-star}, we deduce
\begin{align*}
f^\ast(y)\geq& f^\ast(y_0)+\langle y-y_0,z\rangle+\frac 1 {2L} \|y-y_0\|^2,
\end{align*}
which concludes the proof. 
\end{proof}

\section{Proofs of results}

\subsection{Proof of \texorpdfstring{\Cref{thm:sparsity}}{}}
\label{Section:proof-main}

Fix $K_0\times K_1\Subset \Omega_0\times \Omega_1$. First we lower bound the diameter and Lebesgue measure  of the convex sections $\mathcal S_x$. We recall that by \Cref{thm:interiorC2}, for $\e\leq \e_0$ and a sufficiently small choice of $\e_0>0$, there exists a constant $L>0$, independent of $\varepsilon>0$, such that $\nabla \varphi_\varepsilon$ and $\nabla \psi_\varepsilon$ are $L$-Lipschitz on
\[
K_0 \cup \bigcup_{y\in K_1} \mathcal T_y \Subset \Omega_0
\quad \text{and} \quad
K_1 \cup \bigcup_{x\in K_0} \mathcal S_x\Subset \Omega_1,
\]
respectively.
\begin{lemma}\label[lemma]{lemma:eq:VolumeLower}
 There exists a constant $R_0>0$ such that,  for all $(x,y)\in K_0 \times K_1 $, 
$$ \mathcal S_x\subset \mathbb{B}\left(\nabla \varphi_\e(x), R_0 \e^\frac 1 {d(p-1)+2}\right) \quad \text{and} \quad \mathcal T_y\subset \mathbb{B}\left(\nabla \psi_\e(y), R_0 \e^\frac 1 {d(p-1)+2}\right).$$
As a consequence, for all $(x,y)\in K_0 \times K_1 $, 
\begin{align}\label{eq:VolumeUpper}
|\mathcal S_x|, |{\mathcal T}_y|\leq C_d R_0^d \e^\frac{d}{2+d(p-1)},
\end{align}
where $C_d$ is a dimensional constant. 
\end{lemma}
\begin{proof}
It clearly suffices to prove the claim for $\e\leq \e_0$. 
    Set $x \in K_0$ and let $y_0 \in \partial \psi_\varepsilon^\ast(x)$. 
Using Fenchel's inequality together with \Cref{lem:convex}, we obtain 
$\xi(x,y_0) \ge 0$, and hence $y_0 \in \mathcal S_x$. 
Moreover, by \Cref{lem:convex} we have $x = \nabla \psi_\varepsilon(y_0)$. 
Combining this with the Lipschitz regularity of $\nabla \psi_\varepsilon$, 
we deduce that there exists a constant $C=C(L)>0$ such that
\begin{align}\label{eq:Taylor}
\xi(x,y) 
&= \xi(x,y_0) + \langle x, y-y_0 \rangle 
   + \psi_\varepsilon(y_0) - \psi_\varepsilon(y) \ge \xi(x,y_0) - C \|y-y_0\|^2 .
\end{align}
The monotonicity of $(\cdot)_+^{q-1}$, \eqref{eq:Taylor} and \eqref{eq:Srodinger-system} yield 
\begin{align*}
c(q)\varepsilon^{\,q-1} 
&\ge \int
   \big(\xi(x,y_0)-C \|y-y_0\|^2\big)^{q-1}_+\, \dd\mu(y) \\
&\gtrsim \int_0^\infty 
   \big(\xi(x,y_0)-C r^2\big)_+^{q-1} r^{d-1}\, dr \gtrsim \xi(x,y_0)^{\,q-1+\frac d2}, 
\end{align*}
Rearranging the inequality yields the estimate
\begin{align}\label{eq:max}
\max_{y} \xi(x,y)=\max_{y\in \mathcal S_x} \xi(x,y)
=  \xi(x,y_0)
\lesssim \varepsilon^{\frac{2}{d(p-1)+2}} .
\end{align}
We denote $D(x,y)\coloneqq \psi_\e^\ast(x)+\psi_\e(y)-\langle x,y\rangle.$     Using Fenchel's inequality and \eqref{eq:max}, we derive for $y\in \mathcal S_x$,
\begin{align}\label{eq:DIntroduce}
\xi(x,y) =& \psi_\e^\ast(x)-\varphi_\e(x)-(\psi_\e^\ast(x)+\psi_\e(y)-\langle x,y\rangle)\\
&= \max_{z\in \mathcal S_x} \xi(x,z)-D(x,y)
\leq c \e^\frac 2{2+d(p-1)}-D(x,y),
\end{align}
for some $c>0$. Using the uniform convexity of $\psi_\e^\ast$ from  Theorems~\ref{thm:interiorC2} and \ref{lem:convex}, as well as that $y\in\partial \psi_\e^\ast(\nabla \psi_\eps(y))$, we estimate
\begin{align}\label{eq:DLowerBound}
D(x,y)\geq& \psi_\e^\ast(\nabla \psi_\e(y))+\langle y,x-\nabla \psi_\e(y)\rangle + C \|\nabla \psi_\e(y)-x\|^2+\psi_\e(y)-\langle x,y\rangle\\
=& \psi_\e^\ast(\nabla \psi_\e(y))+\psi_\e(y)-\langle y,\nabla \psi_\e(y) \rangle+C\|\nabla \psi_\e(y)-x\|^2 = C \|\nabla \psi_\e(y)-x\|^2.
\end{align}
The last equality holds due to the equality case of Fenchel's inequality (cf.~\Cref{lem:convex}). Combining \eqref{eq:DLowerBound}, \eqref{eq:DIntroduce} and noting that for $x\in \mathcal T_y$, $\xi(x,y)\geq 0$, we find
\begin{align*}
\max_{x\in \partial \mathcal T_y} \|\nabla \psi_\e(y)-x\|\lesssim \e^\frac 2 {2+d(p-1)}.
\end{align*}
This gives the upper bound for $\mathcal T_y$. An analogous argument provides the upper bound for $\mathcal S_x$.

\end{proof}

In the following lemma we establish the lower bound in \Cref{thm:sparsity}. With this result, the proof of \Cref{thm:sparsity} is complete.
 \begin{lemma}
 There exists a constant $R_0>0$ such that,  for $\e\leq \e_0$ and for all $(x,y)\in K_0 \times K_1 $,
$$ \mathbb{B}\left(\nabla \varphi_\e(x),\frac 1 {R_0} \e^\frac 1 {d(p-1)+2}\right)\subset \mathcal S_x \quad {\rm and} \quad \mathbb{B}\left(\nabla \psi_\e(y),\frac 1 {R_0} \e^\frac 1 {d(p-1)+2}\right)\subset \mathcal T_y.$$
  If $\Omega_1$ is convex, the statement holds for $\e\in (0,1]$.
 \end{lemma}
 \begin{proof} We first assume $\e\leq \e_0$. Set $x_0\in K_0$ and $y_0\in K_1$. Increase $K_1 \Subset \Omega_1$ in such a way that $ \bigcup_{x\in K_0} \mathcal{S}_x \subset K_1 $. 
Since for $x\in K_0$, the function $y\mapsto \xi(x,y)$ is concave on $\mathcal S_x$ and $\nabla \varphi_\e(x) $ is an average (see \eqref{eq:derivative}),  Jensen's inequality yields
\begin{align}\label{eq:Jensen}
\xi(x_0,\nabla \varphi_\e(x_0))\geq \left(\int_{\mathcal S_{x_0}} \xi(x_0,y)^{q-2}\;\dd\mu(y)\right)^{-1} \int_{\mathcal S_{x_0}}\xi(x_0,y)^{q-1}\;\dd\mu(y).
\end{align}
Using \eqref{eq:max} and \eqref{eq:VolumeUpper}, we find
\begin{align}\label{eq:JensenDenominator}
\int_{\mathcal S_{x_0}} \xi({x_0},y)^{q-2}\;\dd\mu(y)\lesssim \max_{y\in \mathcal S_{x_0}} \xi({x_0},y)^{q-2} | \mathcal S_{x_0}|\lesssim \e^\frac{2(q-2)} {2+d(p-1)}\e^\frac{d}{2+d(p-1)}.
\end{align}
Combining \eqref{eq:Jensen}, \eqref{eq:JensenDenominator} and \eqref{eq:Srodinger-system}, we deduce
\begin{align}\label{eq:lowerBoundExact}
\xi({x_0},\nabla \varphi_\e({x_0}))\gtrsim \e^\frac 2 {d(p-1)+2}.
\end{align}
which in particular implies $\nabla \varphi_\eps({x_0})\in \mathcal{S}_{x_0}$.
Since
\begin{align*}
    \xi(x_0,\nabla \varphi_\eps({x_0}))&= \langle {x_0},\nabla \varphi_\eps({x_0}) \rangle  -\varphi_\e({x_0})-\psi_\e(\nabla \varphi_\eps({x_0}))\\
    &=  \xi(x_0,y_0)+\psi_\e(y_0)-\psi_\e(\nabla \varphi_\eps(x_0))+\langle x_0,\nabla \varphi_\e(x_0)-y_0\rangle , 
\end{align*}
 the $L$-Lipschitz regularity of $\nabla\psi_\e$ in $\bigcup_{x\in K_0}\mathcal S_x$ yields 
 $$  \xi(x_0,\nabla \varphi_\eps(x_0))\leq 
     \xi(x_0,y_0)+\langle x_0-\nabla \psi_\e(\nabla\varphi_\e(x_0)),\nabla \varphi_\e(x_0)-y_0\rangle +L \| y_0- \nabla\varphi_\e(x_0)\|^2.  $$
 From the inequality $|ab|\leq \delta a^2 +b^2/(4\delta)$, we derive 
  $$  \xi(x_0,\nabla \varphi_\eps(x_0))\leq 
     \xi(x_0,y_0)+\delta\| x_0-\nabla \psi_\e(\nabla\varphi_\e(x_0))\|^2 +\left(L+\frac{1}{4\delta}\right) \| y_0- \nabla\varphi_\e(x_0)\|^2,  $$
 which, in virtue of $x_0\in \mathcal{T}_{\nabla\varphi_\e(x_0)} $ and \Cref{lemma:eq:VolumeLower}, implies
 \begin{equation}\label{eq:TaylorExact}
      \xi(x_0,\nabla \varphi_\eps(x_0))\leq 
     \xi(x_0,y_0)+\delta  R_0^2 \e^\frac 2 {d(p-1)+2} +\left(L+\frac{1}{4\delta}\right) \| y_0- \nabla\varphi_\e(x_0)\|^2. 
 \end{equation}
Combining  \eqref{eq:lowerBoundExact}  with \eqref{eq:TaylorExact}, choosing $\delta$ sufficiently small, we find $c,C>0$ such that for all $x_0\in K_0$ and $y_0\in K_1\supset \bigcup_{x\in K_0}\mathcal S_x$,
\begin{align}\label{eq:lower-bound}
c\e^\frac 2 {d(p-1)+2}-C \|y_0-\nabla \varphi_\e(x_0)\|^2 \leq \xi(x_0,y_0),
\end{align}
which concludes the statement  for $\mathcal{S}_x$.   An analogous argument provides the inner containment  bound for $\mathcal T_y$.

In the case $\e\geq \e_0$, we comment that $\mathcal S_x\cap \Omega_1$ is a convex set and hence for $x\in K_0$, $y\to \xi(x,y)$ is concave on $\mathcal S_x\cap \Omega_1$. This allows us to carry out the proof as before.
\end{proof}

\subsection{Proof of \texorpdfstring{\Cref{Corollary:InteriorStrongConvexity}}{}}\label{Section:proof-of-corollary}
\begin{proof}[Proof of  \Cref{Corollary:InteriorStrongConvexity}] 
Set $K_0\Subset \Omega_0$ and  $\eps_0>0$  such that $\bigcup_{x\in K_0} \bigcup_{y\in \mathcal{S}_x  } \mathcal{T}_{y}\subset \widetilde{K}_0\Subset \Omega_0 $ for all $\eps\leq \eps_0$. We prove the cases $p=2$, $p\in (1,3/2]$ and $p\in (3/2,2)$ separately. 

{\it Case $p=2$}: 
For $x\in K_0$, note that  $y\in \partial \mathcal{S}_x$ if and only if $\xi(x,y)=0$, which is also equivalent to  $x\in \partial \mathcal{T}_y$. Hence,   $$\|x-\nabla \psi_\eps(y)\|\leq {\rm diam}(\mathcal{T}_y) \lesssim \eps^{\frac{1}{d+2}}. $$ 
Combining this with \eqref{eq:Form2ndDerivative}, we find for almost every $x\in K_0$,
$$ A\coloneqq\langle \nabla^2 \varphi_\e(x) h, h\rangle \geq \frac{\int_{\partial \mathcal{S}_x}  {\left(\langle \nabla \varphi_\e(x)- y, h \rangle \right)^2} \mu(y) d\mathcal{H}^{d-1}(y)}{ \eps^{\frac{1}{d+2}} \mu(\mathcal{S}_x)} , $$
from which, using the estimate $ \mu(\mathcal{S}_x) \lesssim  \eps^{\frac{d}{d+2}}$ and the fact that  $ \inf_{\Omega_1}\mu >0$, we derive 
$$ A \gtrsim \frac{\int_{\partial \mathcal{S}_x }  {\left(\langle \nabla \varphi_\e(x)- y, h \rangle \right)^2} d\mathcal{H}^{d-1}(y)}{ \eps^{\frac{d+1}{d+2}}} . $$
We define the cone $\mathcal{C}_h=\{y:  \langle y, h\rangle \geq \frac{1}{2} \|y\|\}$. Reducing the integration to this cone we get 
\begin{align*}
    A &\gtrsim \frac{\int_{\partial \mathcal{S}_x \cap (\nabla \varphi_\e(x) + \mathcal{C}_h)}  {\left(\langle \nabla \varphi_\e(x)- y, h \rangle \right)^2} d\mathcal{H}^{d-1}(y)}{ \eps^{\frac{d+1}{d+2}}}  \\
    &\gtrsim \frac{\int_{\partial \mathcal{S}_x \cap (\nabla \varphi_\e(x) + \mathcal{C}_h)}  {\| \nabla \varphi_\e(x)- y\|^2} d\mathcal{H}^{d-1}(y)}{ \eps^{\frac{d+1}{d+2}}} \gtrsim  \frac{\mathcal{H}^{d-1}(\partial \mathcal{S}_x \cap (\nabla \varphi_\e(x) + \mathcal{C}_h))}{ \eps^{\frac{d-1}{d+2}}} ,
\end{align*}
where the last estimate follows from  \Cref{thm:sparsity}. To conclude, we define $$\gamma_\eps(y)=  \eps^{-\frac{1}{d+2}} (y-\nabla \varphi_\e(x))   $$ and observe that $$\mathcal{H}^{d-1}(\partial \mathcal{S}_x \cap (\nabla \varphi_\e(x) + \mathcal{C}_h)) = \eps^{\frac{d-1}{d+2}} \mathcal{H}^{d-1}(\gamma_\eps(\partial \mathcal{S}_x \cap (\nabla \varphi_\e(x) + \mathcal{C}_h))) . $$ 
  \Cref{thm:sparsity} yields $$\mathbb{B}(0,\frac{1}{C})\cap \mathcal{C}_h\subset \gamma_\eps(\mathcal{S}_x \cap (\nabla \varphi_\e(x) + \mathcal{C}_h)) \subset \mathbb{B}(0,C) \cap \mathcal{C}_h,  $$
for some $C>0$, irrespective of $\eps$. Hence,  $\mathcal{H}^{d-1}(\gamma_\eps(\partial\mathcal{S}_x \cap (\nabla \varphi_\e(x) + \mathcal{C}_h)))\geq \frac{1}{C'}$ for some $C'>0$ depending just on $C$ and the dimension $d$,  and the result follows for $p=2$.

{\it Case $p\in (1,3/2]$ (i.e., $q\geq 3$).} As before, using \eqref{eq:Form2ndDerivative} and restricting the integration to the cone $\mathcal C_h$, we find for a.e. $x\in K_0$,
$$ A\coloneqq \langle \nabla^2 \varphi_\e(x)h,h\rangle\gtrsim \frac{\int_{ \mathcal S_x \cap (\nabla \varphi_\e(x) +\mathcal{C}_h)}  \|\nabla \varphi_\e(x)- y\|^2 (\xi(x,y))_+^{q-3} d \mu(y)  }{ \int (\xi(x,y))_+^{q-2} d \mu(y)  }.   $$
Further \eqref{eq:lower-bound} yields 
$ c\e^\frac 2 {d(p-1)+2}-C \|y-\nabla \varphi_\e(x)\|^2 \leq \xi(x,y) $, so that 
\begin{align*}
    A&\gtrsim \frac{\int_{ \mathcal S_x\cap (\nabla \varphi_\e(x) +\mathcal{C}_h)}   (\e^\frac 2 {d(p-1)+2}- \|y-\nabla \varphi_\e(x)\|^2)_+^{q-3}  \|y-\nabla \varphi_\e(x)\|^2 \dd \mu(y)  }{ \int (\xi(x,y))_+^{q-2} \dd \mu(y)  } \\
    &\approx \frac{\e^\frac{2q+d-4} {d(p-1)+2}}{ \int (\xi(x,y))_+^{q-2} \dd \mu(y)  } . 
\end{align*}
$$   $$
Use \eqref{eq:max} and  \Cref{thm:sparsity} to derive 
\begin{equation}
    \label{bound-lp-density}
    \int (\xi(x,y))_+^{q-2} \dd \mu(y)  \lesssim  \e^\frac{2(q-2)} {d(p-1)+2}|\mathcal{S}_x| \lesssim     \e^\frac{2(q-2)+d} {d(p-1)+2},
\end{equation}
and conclude.

 {\it Case $p\in (3/2,2)$}
Since $q<3$,  \eqref{eq:max} yields $(\xi(x,y))^{q-3} \gtrsim \e^\frac{2(q-3)} {d(p-1)+2}. $ {\it A fortiori}, restricting to $\mathcal C_h$ as before,
$$ A\gtrsim \frac{\e^\frac{2(q-3)} {d(p-1)+2} \int_{(\nabla \varphi_\e(x) +\mathcal{C}_h)\cap \mathcal{S}_x}   |y-\nabla \varphi_\e(x)|^2 \dd \mu(y)  }{ \int (\xi(x,y))_+^{q-2} \dd \mu(y)  } \approx \frac{\e^\frac{2q+d-4} {d(p-1)+2}  }{ \int (\xi(x,y))_+^{q-2} \dd \mu(y)  } \gtrsim 1, $$
where we also used \Cref{thm:sparsity}. This  concludes the proof.



\end{proof}

\section{Explicit solution for self-transport with Lebesgue marginals}\label{sec:Explicit}
We now specialise to the case of \emph{self-transport} with Lebesgue marginals, i.e.\ $\lambda = \mu = \mathcal{L}$,
the normalised Lebesgue measure on the flat unit torus $\mathbb{T}^d = (\R/\Z)^d$,
equipped with the geodesic distance
$d_{\mathbb{T}^d}(x,y) \coloneqq \min_{k\in\mathbb{Z}^d}\|x-y-k\|$. We will use this to argue that the sparsity rates obtained in \Cref{thm:sparsity} are sharp.
In this setting, the problem \eqref{ROT} reduces to
\begin{equation}\label{ROT-self}
    \inf_{\pi \in \Pi(\mathcal{L},\mathcal{L}),\; \pi \ll \mathcal{L}\otimes\mathcal{L}}
    \int \frac{1}{2}\,d_{\mathbb{T}^d}(x,y)^2 \;\dd\pi
    \;+\;
    \e \int h_p\!\left(\frac{\dd\pi}{\dd(\mathcal{L}\otimes\mathcal{L})}\right)
    \dd(\mathcal{L}\otimes\mathcal{L}) \, ,
\end{equation}
and the identity map $T = \mathrm{Id}$ is the unique OT map from $\mathcal{L}$ to itself.

By the symmetry of \eqref{ROT-self} under the exchange $(x,y)\mapsto (y,x)$,
the dual optimisers $(f_\e, g_\e)$ of \eqref{Dual} satisfy $f_\e = g_\e$
$\mathcal{L}$-almost everywhere.
Substituting into \eqref{eq:density-ROT}, the ROT plan $\pi_\e$ has density
\begin{equation}\label{eq:density-self}
    \rho_\e(x,y)
    = \frac{\dd\pi_\e}{\dd(\mathcal{L}\otimes\mathcal{L})}(x,y)
    = \frac{1}{\e^{q-1}q^{q-1}}
      \Bigl(f_\e(x) + f_\e(y) - \tfrac{1}{2}\,d_{\mathbb{T}^d}(x,y)^2\Bigr)_+^{q-1}.
\end{equation}
Moreover, since the Lebesgue measure $\mathcal{L}$ is invariant under all translations
$\tau_z \colon x \mapsto x+z \pmod{1}$ on $\mathbb{T}^d$,
if  $(f_\e, f_\e)$ is a dual optimizer of \eqref{Dual}, then so is
$(f_\e(\cdot + z), f_\e(\cdot + z))$ for every $z\in\mathbb{T}^d$. As a result, uniqueness of dual optimizers up to an additive constant
(cf.~\cite[Theorem~3.2]{GonzalezSanzEcksteinNutz.25}) forces $f_\e$ to be constant,
$f_\e \equiv C_\e \in \R$.
The density \eqref{eq:density-self} therefore simplifies to
\begin{equation}\label{eq:density-torus}
    \rho_\e(x,y)
    = \frac{1}{\e^{q-1}q^{q-1}}
      \Bigl(2C_\e - \tfrac{1}{2}\,d_{\mathbb{T}^d}(x,y)^2\Bigr)_+^{q-1},
\end{equation}
which depends on $(x,y)$ only through $d_{\mathbb{T}^d}(x,y)$. The constant $C_\e$ is uniquely determined by the marginal constraint,
which via the Schr\"{o}dinger system \eqref{eq:Srodinger-system} reads
\begin{equation}\label{eq:Ce}
    \int_{\mathbb{T}^d}
    \Bigl(2C_\e - \tfrac{1}{2}\,d_{\mathbb{T}^d}(0,z)^2\Bigr)_+^{q-1}
    \dd z
    = \e^{q-1}q^{q-1}.
\end{equation}
For $\e$ small enough, such that $R_\e\coloneqq 2\sqrt{C_\e} < \frac{1}{2}$,
the geodesic ball $\mathbb{B}_{\mathbb{T}^d}(0, R_\e)$ does not intersect
the cut locus of $0$ in $\mathbb{T}^d$ and is isometric to the Euclidean ball
$\mathbb{B}(0, R_\e) \subset \R^d$.
In this setting, the integral in \eqref{eq:Ce} therefore reduces to a Euclidean one,
and converting to polar coordinates gives
\begin{align*}
    \int_{\mathbb{T}^d}
    \Bigl(2C_\e - \tfrac{1}{2}\,d_{\mathbb{T}^d}(0,z)^2\Bigr)_+^{q-1} \dd z
    &= \frac{2\pi^{d/2}}{\Gamma(d/2)}
       \int_0^{R_\e}
       \Bigl(2C_\e - \tfrac{1}{2}r^2\Bigr)^{q-1} r^{d-1} \;\dd r.
\end{align*}
Substituting $r = R_\e t$ and using $R_\e = 2\sqrt{C_\e}$,
\begin{align*}
    &= \frac{2\pi^{d/2}}{\Gamma(d/2)}
       \,(2C_\e)^{q-1}\, R_\e^d
       \int_0^1 (1-t^2)^{q-1} t^{d-1} \;\dd t.
\end{align*}
The remaining integral is evaluated via the substitution $s = t^2$,
\begin{equation*}
    \int_0^1 (1-t^2)^{q-1} t^{d-1}\;\dd t
    = \frac{1}{2}B\!\left(\tfrac{d}{2},\,q\right)
    = \frac{\Gamma(d/2)\,\Gamma(q)}{2\,\Gamma(d/2+q)}.
\end{equation*}
Substituting back and using $R_\e^d = 2^d C_\e^{d/2}$, equation \eqref{eq:Ce} becomes
\begin{equation*}
    \frac{\pi^{d/2}\,\Gamma(q)}{\Gamma(d/2+q)}
    \cdot 2^{q+d-1}
    \cdot C_\e^{q - 1 + d/2}
    = \e^{q-1} q^{q-1},
\end{equation*}
and solving for $C_\e$ yields
\begin{equation}\label{eq:Ce-explicit}
    C_\e
    = \left(
        \frac{\Gamma\!\left(\frac{d}{2}+q\right) q^{q-1}}
             {\,\pi^{d/2}\,\Gamma(q)\,2^{q+d-1}}
      \right)^{\!\frac{1}{q-1+d/2}}
      \e^{\,\frac{2}{d(p-1)+2}},
\end{equation}
where we used $q - 1 = \frac{1}{p-1}$ and
$q - 1 + \frac{d}{2} = \frac{d(p-1)+2}{2(p-1)}$ to simplify the exponent of $\e$.
In particular, $C_\e \approx \e^{\frac{2}{d(p-1)+2}}$ as $\e\to 0$,
and the support radius satisfies $R_\e = 2\sqrt{C_\e} \approx \e^{\frac{1}{d(p-1)+2}}$,
consistent with \Cref{thm:sparsity}.
\begin{remark}
    A similar explicit solution can be found for the self-transport problem on $\mathbb{S}^d$ with the marginals chosen to be the normalised surface measure on $\mathbb{S}^d$. 
\end{remark}

\appendix
\section{Reynolds'  transport theorem}\label{appendix}
For the reader's convenience we sketch here the proof of a slightly non-standard version of Reynolds'  transport theorem.
\begin{theorem}
Let $\Omega_0,\Omega_1\subset \R^d$ be bounded domains, and assume that $\Omega_1$
is Lipschitz. Let
\[
\mathcal R(x):=\{y\in \R^d:\ f(x,y)>0\}.
\]
Assume that:

\begin{itemize}
\item[(i)] for every convex $K_0\subset \Omega_0$ and $K_1\subset \Omega_1$, the map
$f$ is concave in $x$ on $K_0\times \Omega_1$ and concave in $y$ on $\Omega_0\times K_1$;

\item[(ii)] $|\mathcal R(x)\cap \Omega_1|>0$ for every $x\in \Omega_0$;

\item[(iii)] $f\in C^1(U)$ for some open set $U\subset \R^d\times \R^d$
containing
\[
\{(x,y): x\in \Omega_0,\ y\in \overline{\mathcal R(x)\cap \Omega_1}\};
\]

\item[(iv)] $h\in C^1(\overline{\Omega_1})$;

\item[(v)] $\mu$ admits a density, also denoted by $\mu$, which is bounded above and bounded away from
$0$ on $\Omega_1$.
\end{itemize}
For $s\ge 0$, define
\[
g(x):=\int_{\mathcal R(x)\cap \Omega_1} h(y)\,f(x,y)^s\,\dd\mu(y)
=\int_{\mathcal R(x)\cap \Omega_1} h(y)\,f(x,y)^s\,\mu(y)\,\dd y.
\]
Then $g\in W^{1,\infty}_{{\rm loc}}(\Omega_0)$ and, for almost every $x\in \Omega_0$,
\begin{multline}\label{eq:formula-smooth-domains}
   \nabla g(x)
=
\int_{\mathcal R(x)\cap \Omega_1}
h(y)\,s\,f(x,y)^{s-1}\,\nabla_x f(x,y)\,\dd\mu(y)
\\
+
\int_{\partial \mathcal R(x)\cap { \Omega_1}}
h(y)\,f(x,y)_+^s\,\langle V, n \rangle\,\mu(y)\,\dd\mathscr H^{d-1}(y), 
\end{multline}
where $n$ is the unit normal to the moving part of the boundary
$\partial \mathcal R(x)\cap { \Omega_1}$, and $V$ is its velocity field
as $x$ varies. (When $s=0$, the first term is understood to be absent. When $s>0$, the boundary term
vanishes because $f=0$ on $\partial\mathcal R(x)\cap { \Omega_1}$.)
\end{theorem}

\begin{proof}
Fix $x_0\in \Omega_0$. Since $|\mathcal R(x_0)\cap \Omega_1|>0$ and
$\mathcal R(x_0)\cap \Omega_1$ is a Lipschitz domain, we may find a finite number $M$ of strongly star-shaped domains $\omega_i\ni y_i$ such that $\mathcal R(x_0)\cap \Omega_1 = \cup \omega_i$, i.e. there exists $y_i\in \omega_i$, $r_i>0$ such that $\omega_i$ is star-shaped with respect to each $y\in \mathbb B(y_i,r_i)$. By continuity of $f$, after shrinking to a neighborhood $N\Subset\Omega_0$ of $x_0$ there exists $c_0>0$ such that $\min_{i,y\in \mathbb B(y_i,r_i)} f(x,y)\geq c_0$ for all $x\in N$. For $x\in N$, write
\begin{align*}
\Gamma(x)\coloneqq \partial \mathcal R(x)\cap \Omega_1.
\end{align*}
Note that for $x\in N$,  $\mathcal R(x)\cap \Omega_1$ is a Lipschitz set and hence, we may write $\mathcal R(x)\cap \Omega_1=\cup \omega_i(x)$ for a finite collection of $M(x)$ star-shaped domains $\omega_i$. Since the map $x\to \mathcal R(x)\cap \Omega_1$ is continuous with respect to Hausdorff distance, we may ensure, upon reducing $N$ if required, that $M(x)=M$ and moreover, $B(y_i,r_i/2)\in \omega_i(x)$ for each $x\in N$ and $i\leq M$. 

Now, for $y\in \Gamma(x)$ let $i$ be such that $y\in \omega_i(x)$. Note that $Z(y)=\cup_{z\in B(y_i,r_i/2)} [z,y]\subset \omega_i$ is a convex subset of $\Omega_1$ containing $y$ and $y_i$. Consequently, due to the concavity of $y\to f(x,y)$ on $Z(y)$,
\[
f(x,y_i)-f(x,y)\le \langle \nabla_y f(x,y), y_i-y \rangle.
\]
Because $f(x,y)=0$ on $\Gamma(x)$ and $f(x,y_i)\ge c_0$, we obtain
\[
c_0\le \langle \nabla_y f(x,y), y_i-y \rangle \le \|\nabla_y f(x,y)\|\,\|y_i-y\|.
\]
Since $\Omega_1$ is bounded, $\|y_i-y\|\le \operatorname{diam}(\Omega_1)<\infty$, and therefore
\[
|\nabla_y f(x,y)|\ge c_1>0
\qquad\text{for all }x\in N,\ y\in \Gamma(x),
\]
for some constant $c_1$ depending only on $N$, $\{y_i\}$, $\{r_i\}$ and $\Omega_1$. In particular, the level set $\Sigma:=\{(x,y)\in N\times \Omega_1:\ f(x,y)=0\}$
is, by the implicit function theorem, a $C^1$ hypersurface in
$N\times \Omega_1$. Thus $\Gamma(x)$ is a $C^1$ hypersurface in $\Omega_1$,
depending $C^1$ on $x$.

Because $\nabla_y f$ is continuous and bounded away from $0$ on $\Sigma$,
after shrinking $N$ if needed there exist $\tau>0$ and constants $0<c_2\le C_2<\infty$
such that for every $x\in N$ and every
$y\in \mathcal R(x)\cap \Omega_1$ with $\operatorname{dist}(y,\Gamma(x))<\tau$,
\begin{equation}\label{eq:bound-distance-f}
    c_2\,\operatorname{dist}(y,\Gamma(x))
\le f(x,y)\le
C_2\,\operatorname{dist}(y,\Gamma(x)).
\end{equation}
Indeed, this follows from the mean value theorem along the normal segment joining
$y$ to its projection onto $\Gamma(x)$.
Fix an index $i\in\{1,\dots,d\}$. Now we claim that 
\begin{equation}
    \label{eq:Bound-Ai}
A_i(x):=
\int_{\mathcal R(x)\cap \Omega_1}
|h(y)|\,s\,f(x,y)^{s-1}\,|\partial_{x_i}f(x,y)|\,\dd\mu(y)
<\infty,
\end{equation}
for every $x\in N$ when $s>0$. We derive 
\begin{align*}
    A_i(x)=
&\int_{\mathcal R(x)\cap \Omega_1 \cap \{ \operatorname{dist}(y,\Gamma(x))\geq \tau\}}
|h(y)|\,s\,f(x,y)^{s-1}\,|\partial_{x_i}f(x,y)|\,\dd\mu(y) \\
&\quad+ \int_{\mathcal R(x)\cap \Omega_1 \cap \{ \operatorname{dist}(y,\Gamma(x))< \tau\}}
|h(y)|\,s\,f(x,y)^{s-1}\,|\partial_{x_i}f(x,y)|\,\dd\mu(y)\\
&\leq \|h\|_\infty s \tau^{s-1} \|\nabla f\|_\infty + \|h\|_\infty s \tau^{s-1} \|\nabla f\|_\infty  \int_{\mathcal R(x)\cap \Omega_1 \cap \{ \operatorname{dist}(y,\Gamma(x))< \tau\}}
f(x,y)^{s-1}\,\dd\mu(y) 
\end{align*}
and obtain \eqref{eq:Bound-Ai} as a consequence of \eqref{eq:bound-distance-f}. 

Next, consider the boundary term. Let $n$ be the unit normal to $\Gamma(x)$.
If $\gamma(t)$ is a curve on $\Sigma$ representing the motion of the free boundary when
$x_i$ varies, then differentiating the identity
$f(x+t e_i,\gamma(t))=0$ at $t=0$ yields
\[
\partial_{x_i}f(x,y)+\langle \nabla_y f(x,y), V_i(x,y) \rangle=0,
\]
where $V_i$ is the velocity of the boundary when $x_i$ varies.
Taking the scalar product with $n$ gives
\[
\langle V_i, n\rangle
=
-\frac{\partial_{x_i}f}{\langle \nabla_y f, n \rangle}.
\]
Since $|\langle \nabla_y f, n \rangle|=\|\nabla_y f\|\ge c_1$ on $\Gamma(x)$, we obtain $|\langle V_i, n \rangle|\le C$ on $\Gamma(x)$
for a constant $C$ independent of $x\in N$.
Hence
\begin{equation}
    \label{eq:bound-B-i}
    B_i(x):=
\int_{\Gamma(x)}
|h(y)|\,f(x,y)_+^s\,|\langle V_i, n\rangle|\,\mu(y)\,\dd\mathscr H^{d-1}(y)<\infty.
\end{equation}
(For $s>0$, this is in fact identically zero since $f=0$ on $\Gamma(x)$.)

Choose a sequence  of smooth sets $\{\Omega_1^m\}_m$ such that $\Omega_1^m\subset \Omega_1^{m+1} \Subset \Omega_1$ and $\bigcup_{m=1}^\infty \Omega_1^m=\Omega_1.$
Set
\[
D_m(x):=\mathcal R(x)\cap \Omega_1^m,
\qquad
g_m(x):=\int_{D_m(x)} h(y)\,f(x,y)^s\,d\mu(y).
\]
The moving part of $\partial D_m(x)$ is
\[
\Gamma_m(x):=\Gamma(x)\cap \Omega_1^m,
\]
while the part on $\partial\Omega_1^m$ is fixed in $x$. Because the integrability estimates \eqref{eq:Bound-Ai} and \eqref{eq:bound-B-i} hold uniformly on compact subsets of $N$,
dominated convergence yields, for every $x\in N$, 
\begin{equation}\label{eq:convergence-in-m-1}
g_m(x)\longrightarrow g(x), \qquad \text{as }\  m\to \infty,   
\end{equation}
and similarly
\begin{multline}\label{eq:convergence-in-m-2}
      \int_{D_m(x)}
h(y)\,s\,f(x,y)^{s-1}\,\partial_{x_i}f(x,y)\,\dd\mu(y)
\\\longrightarrow
\int_{\mathcal R(x)\cap \Omega_1}
h(y)\,s\,f(x,y)^{s-1}\,\partial_{x_i}f(x,y)\,\dd\mu(y),
\end{multline}
as well as
\begin{multline}\label{eq:convergence-in-m-3}
    \int_{\Gamma_m(x)}
h(y)\,f(x,y)_+^s\,\langle V_i, n\rangle\,\mu(y)\,\dd\mathscr H^{d-1}(y)
\\
\longrightarrow
\int_{\Gamma(x)\setminus \partial\Omega_1}
h(y)\,f(x,y)_+^s\,\langle V_i, n\rangle\,\mu(y)\,\dd\mathscr H^{d-1}(y),
\end{multline}
as $m\to \infty$ and 
for all $x\in N$. Therefore it is enough to prove \eqref{eq:formula-smooth-domains} on each smooth domain $\Omega_1^m$.
So, from now on, we may assume that the ambient domain is smooth.

Fix $m$ and $i\in\{1,\dots,d\}$.
Since the moving boundary $\Gamma_m(x)$ is $C^1$ in $x$ and the remaining part of
$\partial D_m(x)$ is fixed, Reynolds’ transport theorem (cf.~\cite[Appendix C.4.]{Evans}) gives
\[
\partial_{x_i} g_m(x)
=
\int_{D_m(x)}
\partial_{x_i}\!\big(h(y)\,f(x,y)^s\,\mu(y)\big)\,\dd y
+
\int_{\Gamma_m(x)}
h(y)\,f(x,y)_+^s\,\langle V_i, n\rangle\,\mu(y)\,\dd\mathscr H^{d-1}(y).
\]
Since $h$ and $\mu$ do not depend on $x$, we have, for $s>0$,
\[
\partial_{x_i}\!\big(h(y)\,f(x,y)^s\,\mu(y)\big)
=
h(y)\,s\,f(x,y)^{s-1}\,\partial_{x_i}f(x,y)\,\mu(y).
\]
Thus
\begin{multline}\label{eq:formula-g-m}
    \partial_{x_i} g_m(x)
=
\int_{D_m(x)}
h(y)\,s\,f(x,y)^{s-1}\,\partial_{x_i}f(x,y)\,\dd\mu(y)
\\+
\int_{\Gamma_m(x)}
h(y)\,f(x,y)_+^s\,\langle V_i, n\rangle\,\mu(y)\,\dd\mathscr H^{d-1}(y).
\end{multline}
Set $\varphi\in C_c^\infty(N)$.
For every $m\in \NN$,
\[
-\int_N g_m(x)\,\partial_{x_i}\varphi(x)\,\dd x
=
\int_N \varphi(x)\,\partial_{x_i}g_m(x)\,\dd x.
\]
Substituting \eqref{eq:formula-g-m}  and using the limits \eqref{eq:convergence-in-m-1}, \eqref{eq:convergence-in-m-2} and \eqref{eq:convergence-in-m-3},
together with dominated convergence, we obtain
\begin{align*}
-\int_N g(x)\,\partial_{x_i}\varphi(x)\,\dd x
&=
\int_N \varphi(x)
\int_{\mathcal R(x)\cap \Omega_1}
h(y)\,s\,f(x,y)^{s-1}\,\partial_{x_i}f(x,y)\,\dd\mu(y)\,\dd x \\
&\quad+
\int_N \varphi(x)
\int_{\partial \mathcal R(x)\setminus \partial\Omega_1}
h(y)\,f(x,y)_+^s\,\langle V_i, n\rangle\,\mu(y)\,\dd\mathscr H^{d-1}(y)\,\dd x.
\end{align*}
Hence, the distributional derivative $\partial_{x_i}g$ is represented by the right-hand
side above. Since this holds for every $i$, we conclude that
$g\in W^{1,1}_{\rm loc}(N)$ and
\begin{multline*}
    \partial_{x_i} g(x)
=
\int_{\mathcal R(x)\cap \Omega_1}
h(y)\,s\,f(x,y)^{s-1}\,\partial_{x_i}f(x,y)\,\dd\mu(y)
\\+
\int_{\partial \mathcal R(x)\setminus \partial\Omega_1}
h(y)\,f(x,y)_+^s\,\langle V_i, n\rangle\,\mu(y)\,\dd\mathscr H^{d-1}(y),
\end{multline*}
for almost every $x\in N$. Since $x_0\in \Omega_0$ was arbitrary, we conclude that $g\in W^{1,1}_{\rm loc}(\Omega_0)$. From the estimates \eqref{eq:Bound-Ai} and \eqref{eq:bound-B-i}, we derive $g\in W^{1,\infty}_{\rm loc}(\Omega_0)$. The result follows.  
\end{proof}

\bibliographystyle{amsalpha}
\bibliography{biblio}

 \end{document}